\documentclass[10pt,oneside]{amsart}
\usepackage{amscd,amssymb, bbm,color}
\usepackage{graphicx, amsmath, amssymb, amscd, amsthm, euscript, psfrag, amsfonts,bm}

\usepackage{amscd,amssymb}

\usepackage[all,arc]{xy}
\usepackage{enumerate}
\usepackage{mathrsfs}
\usepackage{amsmath}
\usepackage{graphicx}
\usepackage{caption}
\usepackage{subcaption}
\usepackage{amssymb}

\newtheorem{thm}{Theorem}[section]

\newtheorem{lem}[thm]{Lemma}

\theoremstyle{definition}
\newtheorem{defn}[thm]{Definition}

\newtheorem{exmp}[thm]{Example}

\theoremstyle{remark}

\newcommand{\be}{begin{equation}}

\newcommand{{\grinv}}{{\Cal G}^{-r}}

\newcommand{\Cal}{\mathcal}

\newcommand{\bp}{\begin{pmatrix}}
\newcommand{\ep}{\end{pmatrix}}

\renewcommand{\bp}{{\rm bp}}

\newcommand{\T}{\operatorname{T}}

\newcommand{\op}{\operatorname}

\renewcommand{\be}{\begin{equation}}
\newcommand{\ee}{\end{equation}}

\renewcommand{\T}{\op{T}}

\newcommand{\Hom}{\op{H}}

\newcommand{\un}{\underline}

\makeatletter
\let\c@equation\c@thm
\makeatother
\numberwithin{equation}{section}
\title[Local Mixing on abelian covers] {Local Mixing on abelian covers of hyperbolic surfaces with cusps}

\author{Wenyu Pan}
\address{Pennsylvania State University, State College, PA, 16802,  USA}
\email{wup60@psu.edu}
\address{Current address: Department of Mathematics, University of Chicago, Chicago, IL, 60637}
\email{wenyu@math.uchicago.edu}

\begin{document}

\maketitle

\begin{abstract} 
We prove the local mixing theorem for geodesic flows on abelian covers of finite volume hyperbolic surfaces with cusps, which is a continuation of the work \cite{OP}. We also describe applications to counting problems and the prime geodesic theorem. 
 \end{abstract}

\section{Introduction}
\subsection{Setting and main results}

Let $M$ be a hyperbolic surface. We can present $M$ as the quotient $\Gamma\backslash \mathbb{H}^2$ of the hyperbolic $2$-space for some torsion-free discrete subgroup $\Gamma$ in $G=\operatorname{PSL}_2(\mathbb{R})$. The unit tangent bundle $\T^1(M)$ is isomorphic to $\Gamma\backslash G$ and the geodesic flow on $\operatorname{T}^1(M)$ corresponds to the right translation action of $a_t=\begin{pmatrix} e^{t/2} & 0 \\ 0 & e^{-t/2}\end{pmatrix}$ on $\Gamma\backslash G$.

Set $\mathsf{m}^{\operatorname{Haar}}_{\Gamma}$ to be the $G$-invariant measure on $\Gamma\backslash G$, which (up to a scalar) equals the hyperbolic volume measure on $\T^1(M)$ when identifying $\Gamma\backslash G$ with $\T^1(M)$. When there is no ambiguity about the group $\Gamma$, we write $\mathsf{m}^{\operatorname{Haar}}$ for simplicity. For any $\psi_1,\psi_2\in C_c(\Gamma\backslash G)$, consider the matrix coefficient 
\begin{equation*}
\langle a_t\cdot \psi_1,\psi_2\rangle=\int_{\Gamma\backslash G} \psi_1(xa_t) \psi_2(x)d\mathsf{m}^{\operatorname{Haar}}(x).
\end{equation*}
One of the central questions in homogeneous dynamics is what we can say about the asymptotic behavior of $\langle a_t\cdot \psi_1,\psi_2\rangle$ as $t\to \infty$. When $\mathsf{m}^{\operatorname{Haar}}(\Gamma\backslash G)<\infty$, it is well-known that 
\begin{equation*}
\lim_{t\to \infty} \langle a_t\cdot \psi_1,\psi_2\rangle=\frac{1}{\mathsf{m}^{\operatorname{Haar}}(\Gamma\backslash G)} \mathsf{m}^{\operatorname{Haar}}(\psi_1) \mathsf{m}^{\operatorname{Haar}}(\psi_2).
\end{equation*}
So the geodesic flow on $\operatorname{T}^1(\Gamma\backslash \mathbb{H}^2)$ satisfies the \emph{strong mixing property}. When $\mathsf{m}^{\operatorname{Haar}}(\Gamma\backslash G)=\infty$, we have
\begin{equation*}
\lim_{t\to \infty}\langle a_t\cdot \psi_1,\psi_2\rangle=0
\end{equation*}
(see for example \cite{HM}).
In view of this, the question of understanding $\langle a_t\cdot \psi_1,\psi_2 \rangle$ in infinite volume setting can be formulated more precisely as whether there exists a renormalization function $\alpha:\mathbb{R}\to \mathbb{R}_{>0}$ such that for any $\psi_1,\psi_2$, we have
\begin{equation*}
\alpha(t)\langle a_t\cdot \psi_1,\psi_2\rangle\to \text{non-trivial}\,\,\,\text{as}\,\,\,t\to \infty.
\end{equation*}
If such renormalization function exists, the $a_t$-action on $\Gamma\backslash G$ (or the geodesic flow on $\operatorname{T}^1(\Gamma\backslash \mathbb{H}^2)$) is said to satisfy the \emph{local mixing property}, a property introduced in \cite{OP} to substitute the strong mixing property in infinite volume setting.

In this paper, we investigate the matrix coefficients/ local mixing property of the geodesic flows on \emph{abelian covers of finite volume hyperbolic surfaces with cusps}.  It is a follow-up work of \cite{OP} and the extensive study by Ledrappier and Sarig \cite{LS} on this kind of surfaces has laid a solid foundation for us. Throughout the paper, set $\Gamma_0$ to be a torsion-free \emph{non-uniform} lattice in $G$ and 
\begin{equation*}
\Gamma\triangleleft\Gamma_0
\end{equation*}
to be a normal subgroup with $\mathbb{Z}^d$-quotient. Then $M=\Gamma\backslash \mathbb{H}^2$ is a regular cover of $M_0=\Gamma_0\backslash \mathbb{H}^2$ whose group of deck transformations is isomorphic to $\mathbb{Z}^d$. Our main result is as follows:
\begin{thm}
\label{main thm}
For any $\psi_1,\psi_2\in C_c(\Gamma\backslash G)$,
\begin{equation*}
\lim_{t\to +\infty}t^{p+h/2}\int_{\Gamma\backslash G} \psi_1(xa_t)\psi_2(x)d\mathsf{m}^{\operatorname{Haar}}(x)=\mathsf{c}\,\mathsf{m}^{\operatorname{Haar}}(\psi_1)\mathsf{m}^{\operatorname{Haar}}(\psi_2),
\end{equation*}
where $p,h\in \mathbb{N}$ and $\mathsf{c}>0$ are some constants given in (\ref{con 1}) and (\ref{con 2}) and $p+h=d$.
\end{thm}

\begin{exmp} Let $$\Gamma(2)=\left\{\begin{pmatrix} a & b\\ c & d\end{pmatrix}\in \operatorname{PSL}_2(\mathbb{Z}): \begin{pmatrix} a& b\\ c & d\end{pmatrix}\equiv \begin{pmatrix} 1 & 0\\ 0 & 1\end{pmatrix}(\operatorname{mod} 2)\right\},$$ which is  a principal congruent subgroup of the modular group $\operatorname{PSL}_2(\mathbb{Z})$. It is a free group generated by $\begin{pmatrix} 1 & 2\\ 0 & 1\end{pmatrix}$  and $\begin{pmatrix} 1 & 0\\ 2 & 1\end{pmatrix}$.The quotient $\Gamma(2)\backslash \mathbb{H}^2$ is the thrice punctured sphere equipped with the hyperbolic metric. Let $\Gamma\backslash \mathbb{H}^2$ be the homology cover of  $\Gamma(2)\backslash \mathbb{H}^2$, that is to say, the group of deck transformations is isomorphic to $\operatorname{H}_1(\Gamma(2)\backslash \mathbb{H}^2;\mathbb{Z})$. As $\operatorname{H}_1(\Gamma(2)\backslash \mathbb{H}^2;\mathbb{Z})\cong \mathbb{Z}^2$, $\Gamma\backslash \mathbb{H}^2$ is a $\mathbb{Z}^2$-cover of $\Gamma(2)\backslash \mathbb{H}^2$. Applying Theorem \ref{main thm} to $\Gamma\backslash G$, we have $p=2, h=0$ and
\begin{equation*}
\lim_{t\to \infty} t^2\int_{\Gamma\backslash G} \psi_1(xa_t)\psi_2(x)d\mathsf{m}^{\operatorname{Haar}}(x)=\mathsf{c}\,\mathsf{m}^{\operatorname{Haar}}(\psi_1) \mathsf{m}^{\operatorname{Haar}}(\psi_2)
\end{equation*}
for any $\psi_1,\psi_2\in C_c(\Gamma\backslash G)$. 
\end{exmp}

Assume that $M$ is a genus $g$ surface with $t+1$ cusps. In general, when $M$ is a homology cover of $M_0$, we have  $\Gamma\backslash \Gamma_0\cong \operatorname{H}_1(M_0;\mathbb{Z})\cong \mathbb{Z}^{2g+t}$, $h=2g$ and $p=t$.

In finite volume case, strong mixing property of geodesic flow and its effective refinement have been successfully applied to address some counting and equidistribution problems \cite{Ma, EM, DRS}. In abelian cover case, we obtain similar applications using local mixing property.

\subsection{Orbit counting} 

\begin{thm}
For any $x,y\in \mathbb{H}^2$, we have
\begin{equation*}
\#\{\gamma\in \Gamma:\,d(x,\gamma y)<T\}\sim \mathsf{c}\frac{e^T}{T^{p+h/2}}\footnote{We write $f(T)\sim g(T)$ if $\lim_{T\to \infty} f(T)/g(T)=1$.},
\end{equation*}
where $d(\cdot,\cdot)$ is the hyperbolic metric on $\mathbb{H}^2$.
\end{thm}
When $\Gamma_0$ is a uniform lattice, an analogous result is due to Pollicott-Sharp \cite{PS} (see also \cite{OP}). 

\begin{proof}
Consider the Cartan decomposition of $G=KA^+K$, where $K=\operatorname{SO}_2(\mathbb{R})/\{\pm I\}=\operatorname{Stab}_G(i)$ (here we are using the upper half plane model for $\mathbb{H}^2$) and $A^+=\{a_t:t\geq 0\}$. For every $T>1$, set $B_T=\{k_1a_tk_2:\,k_1,k_2\in K,\,a_t\,\text{wtih}\,1< t<T\}$. Define the counting function on $\Gamma\backslash G\times \Gamma \backslash G$ by
\begin{equation*}
F_{B_T}(h_1,h_2)=\sum_{\gamma\in \Gamma}1_{B_T}(h_1^{-1}\gamma h_2).
\end{equation*}
For $x,y\in \mathbb{H}^2$, write $x=g_1 i$ and $y=g_2 i$. We have
\begin{equation*}
\#\{\gamma\in \Gamma:\,d(x,\gamma y)<T\}=F_{B_T}(g_1,g_2)+O(1).
\end{equation*}

For $i=1,2$, let $\psi_i^{\epsilon}\in C^{\infty}(G)$ be an $\epsilon$-approximation of $g_i$, i.e., $\psi_i^{\epsilon}$ is a non-negative smooth function supported on the $\epsilon$-neigborhood of $g_i$ and $\int \psi_i^{\epsilon}dg=1$,  and let $\Psi^{\epsilon}_i\in C^{\infty}(\Gamma\backslash G)$ be its $\Gamma$-average: $\Psi_i^{\epsilon}(\Gamma g)=\sum_{\gamma\in \Gamma}\psi_i^{\epsilon}(\gamma g)$. By folding and unfolding the integrals, we have the following
\begin{align}
\label{folding}
&\langle F_{B_T},\Psi_1^{\epsilon}\otimes \Psi_2^{\epsilon}\rangle_{\Gamma\backslash G \times \Gamma\backslash G}\\
=&\int_{h\in B_T} \int_{g\in \Gamma\backslash G} \Psi_1^{\epsilon}(g)\Psi_2^{\epsilon}(gh)dg dh\nonumber\\
=&\int_{k_1a_tk_2\in B_T} \int_{g\in \Gamma\backslash G} \Psi_1^{\epsilon}(g) \Psi_2^{\epsilon}(gk_1a_tk_2)e^t dg dk_1 dt dk_2\nonumber \\
=&\int_{k_1a_tk_2\in B_T} \int_{g\in \Gamma\backslash G}\Psi_1^{\epsilon}(gk_1^{-1})\Psi_2^{\epsilon}(ga_tk_2)e^{t}dgdk_1dtdk_2.\nonumber
\end{align}
Using Theorem \ref{main thm}, we obtain
\begin{align*}
&\langle F_{B_T},\Psi_1^{\epsilon}\otimes \Psi_2^{\epsilon}\rangle_{\Gamma\backslash G\times \Gamma\backslash G}\\
=&\int_{k_1a_tk_2\in B_T} \frac{e^t}{t^{p+h/2}}(1+o(1))\mathsf{m}^{\operatorname{Haar}}(k_1^{-1}\cdot \Psi_1^{\epsilon}) \mathsf{m}^{\operatorname{Haar}}(k_2\cdot \Psi_2^{\epsilon}) dt dk_1k_2=\mathsf{c}\frac{e^{T}}{T^{p+h/2}}(1+o(1)).
\end{align*}
The strong wave front  property for $KA^+K$-decomposition \cite{GOS} implies that
\begin{equation*}
F_{B_T}(g_1,g_2)\approx \langle F_{B_T},\Psi_1^{\epsilon}\otimes \Psi_2^{\epsilon}\rangle_{\Gamma\backslash G\times \Gamma \backslash G},
\end{equation*}
finishing the proof of the theorem. 
\end{proof}

\subsection{Prime geodesic theorem} For $T>0$, let $\mathcal{P}_T$ be the collection of all primitive closed geodesics in $\operatorname{T}^1(M)$ of length at most $T$. Note that $\# \mathcal{P}_T=\infty$ if $d\geq 1$.

\begin{thm}
\label{geod}
Let $\Omega\subset \T^1(M)$ be a compact subset with Haar-negligible boundary. Then as $T\to \infty$,
\begin{equation}
\label{geod1}
\sum_{C\in \mathcal{P}_T}\frac{l(C\cap \Omega)}{l(C)}\sim\frac{\mathsf{c} \,e^T}{T^{p+h/2+1}} \mathsf{m}^{\operatorname{Haar}}(\Omega) .
\end{equation}
\end{thm}

Given Theorem \ref{main thm}, the proof of \cite[Theorem 5.1]{MMO} applies in the same way. The idea goes back to Margulis' work \cite{Ma}. The key is to estimate
\begin{equation}
\label{geod2}
\sum_{C\in \mathcal{P}_T}l(C\cap \Omega).
\end{equation}
From \eqref{geod2} to \eqref{geod1}, one uses a soft idea: Abel's summation formula (see the proof of Theorem 5.1 in \cite{MMO}).  To estimate \eqref{geod2}, one of the main observations is to relate \eqref{geod2} with a $\Gamma$-orbit counting problem \cite[Lemma 5.14]{MMO}. In the current setting, $\operatorname{T}^1(M)$ is not compact and this brings the challenge of estimating
\begin{equation}
\label{geod3}
\sum_{C\in \mathcal{P}_T}l(C\cap \text{infinity of}\,\,\,\operatorname{T}^1(M)).
\end{equation} 
When $M$ is a geometrically finite manifold, infinity of $\operatorname{T}^1(M)$ is the unit tangent bundle over the cuspidal regions. Such estimate is obtained by Roblin \cite{Ro} (see also \cite[Theorem 5.20]{MMO} for the statement). In \cite{Ve}, the author proved a prime geodesic theorem for SPR manifolds which are a generalization of geometrically finite manifolds with cusps \cite[Theorem 1.5]{Ve}, and the proof is built on a detailed study of the thermodynamics formalism of the geodesic flow on non-compact pinched negatively curved manifolds. When $M$ is a $\mathbb{Z}^d$-cover, infinity of $\operatorname{T}^1(M)$ includes the part with large $\mathbb{Z}^d$-coordinate, which is rather different from the cuspidal regions. It is not clear how to obtain \eqref{geod3} for such infinity. 

Another formulation of the prime geodesic theorem would be studying the distribution of closed geodesics in $\operatorname{T}^1(M_0)$ satisfying some homological constraints. An explicit main term in our setting was obtained by Epstein \cite{Ep} by exploring the spectral theory of the Laplace operator acting on flat line bundles and relating this with the counting function using Selberg trace formula. 

\subsection{Constants in the main theorem}
\label{subspace}
Let $\mathcal{H}\subset \Hom^1(M_0;\mathbb{R})$ be the linear subspace of cohomology classes which vanish on projections of cycles in $\Hom_1(M;\mathbb{R})$ to $\Hom_1(M_0;\mathbb{R})$. Since $M$ is a $\mathbb{Z}^d$-cover of $M_0$, the dimension of $\mathcal{H}$ is $d$. 

We describe a basis for $\mathcal{H}$. Given a loop $c$ in $M_0$, let $\tilde{c}$ be a lift of $c$ in $M$ and $g_c\in \Gamma_0$ the isometry mapping the beginning of $\tilde{c}$ to its endpoint. This defines a map
\begin{equation*}
\text{Frob}:\{\text{loops in}\,M_0\}\to \Gamma\backslash \Gamma_0,\,\,\,c\mapsto \Gamma g_c.
\end{equation*}
Observe that $\operatorname{Frob}(\cdot)$ depends only on the homotopy classes. So $\operatorname{Frob}(\cdot)$ is a homomorphism from $\pi_1(M_0)$ to $\Gamma\backslash \Gamma_0$. Since $\Gamma\backslash \Gamma_0$ is abelian, $\operatorname{Frob}(\cdot)$ can be further regarded as a homomorphism from $\Hom_1(M_0; \mathbb{Z})$ to $\Gamma\backslash \Gamma_0$. Using the isomorphism $\Gamma\backslash \Gamma_0\cong \mathbb{Z}^d$, the maps
\begin{equation*}
[c]\in \Hom_1(M_0;\mathbb{Z})\mapsto \langle \text{Frob}([c]),e_i\rangle\,\,\, (\{e_i\}\,\,\,\text{is the standard basis in}\,\,\,\mathbb{R}^d)
\end{equation*}
yield $d$ linearly independent cohomology classes in $\mathcal{H}$, which form a basis of $\mathcal{H}$.

We represent the elements in $\mathcal{H}$ by real harmonic $1$-forms with at most simple poles at the cusps (this is possible, see \cite[Section 2]{GL}, \cite[Lemma 1, 2]{DoSa} for example). The residue of a $1$-form at a cusp is the integral of that form on a loop which is homotopic to the cusp. Decompose $\mathcal{H}$ into a direct sum 
\begin{equation*}
\mathcal{H}=\mathcal{H}_{\text{p}}\oplus \mathcal{H}_{\text{h}},
\end{equation*}
where 
\begin{align*}
&\mathcal{H}_{\text{h}}:=\{w\in \mathcal{H}:\,\,\text{the residues of}\,\,w\,\text{at the cusps are all zero}\},\\
&\mathcal{H}_{\text{p}}:=\mathcal{H}/\mathcal{H}_{\text{h}}.
\end{align*}
The constants $p$ and $h$ in Theorem \ref{main thm} are given by
\begin{equation}
\label{con 1}
p=\text{dim}\, \mathcal{H}_{\text{p}}\,\,\,\text{and}\,\,\,h=\text{dim}\, \mathcal{H}_{\text{h}}.
\end{equation}

Set  $\overline{\mathsf{m}}^{\operatorname{Haar}}_{\Gamma_0}$ to be the measure on $M_0$ induced by $\mathsf{m}^{\operatorname{Haar}}_{\Gamma_0}$  on $\T^1(M_0)$ and \begin{equation}
\label{area}
\mathsf{m_0}=\overline{\mathsf{m}}^{\operatorname{Haar}}_{\Gamma_0}(M_0).
\end{equation} 
Assume $M_0$ is a genus $g$ surface with $t+1$ cusps. For $w\in \mathcal{H}$, denote by $\lambda_1(w),\ldots,\lambda_{t+1}(w)$ the residues of $w$ at $t+1$ cusps of $M_0$. Denote by $\lVert \cdot\rVert$ the norm in the cotangent bundle. Equip $\mathcal{H}_{\text{p}},\mathcal{H}_{\text{h}}$ with the following norms
\begin{align*}
&\lVert w\rVert_{\text{p}}:=\frac{1}{\mathsf{m}_0} \sum_{j=1}^{t+1}|\lambda_j(w)|\,\,\,\,\,\,\,\,\text{for}\,\,\,w\in \mathcal{H}_{\text{p}},\\
&\lVert w\rVert_{\text{h}}:=\left(\frac{1}{\mathsf{m}_0}\int_{M_0}\lVert w\rVert^2 d \overline{\mathsf{m}}_{\Gamma_0}^{\operatorname{Haar}}\right)^{1/2}\,\,\,\,\,\,\,\,\text{for}\,\,\,w\in \mathcal{H}_{\text{h}}.
\end{align*}

We identify $\mathcal{H}$ with $\mathbb{R}^d$ using the basis for $\mathcal{H}$ described above. Let $\mathbb{E}_{\text{p}}, \mathbb{E}_{\text{h}}$ be the linear subspaces in $\mathbb{R}^d$ corresponding to $\mathcal{H}_{\text{p}}$ and $\mathcal{H}_{\text{h}}$ respectively under the this identification. The norm on $\mathcal{H}_{\text{p}}$ (resp. $\mathcal{H}_{\text{h}}$) induces a norm on $\mathbb{E}_{\text{p}}$ (resp. $\mathbb{E}_{\text{h}}$), which we still denote by $\lVert \cdot \rVert_{\text{p}}$ (resp. $\lVert \cdot \rVert_{\text{h}}$). The constant $\mathsf{c}$ in Theorem \ref{main thm} is given by
\begin{equation}
\label{con 2}
\mathsf{c}=\frac{1}{(2\pi)^d\cdot\mathsf{m}_0}\int_{\mathbb{E}_{\text{p}}} e^{-\lVert \underline{x}\rVert_{\text{p}}}d\underline{x} \int_{\mathbb{E}_{\text{h}}} e^{-\lVert \underline{y}\rVert^2_{\text{h}}}d\underline{y}.
\end{equation}

\subsection{About the proof of Theorem \ref{main thm}} Theorem \ref{main thm} will be proved in pp. 11 -12 after preparatory work. It is proved following the strategy in \cite{OP}, which is inspired by the work of Guivarc'h and Hardy \cite{GH}. One uses the Stadlbauer-Ledrappier-Sarig's coding to  model the geodesic flow over $\T^1(M_0)$ by the suspension flow over the suspension space
\begin{equation*}
\Sigma^{\tau}:=\Sigma\times \mathbb{R}/\sim,
\end{equation*}
where $(\Sigma,\sigma)$ is a topologically mixing two-sided Markov shift of \emph{countable states} and the equivalence relation is defined via the left shift map $\sigma$ and the first return time $\tau:\Sigma\to \mathbb{R}$.  Then one lifts this coding to the covering space and hence the geodesic flow on $\operatorname{T}^1(M)$ can be studied through the suspension flow over the suspension space
\begin{equation*}
\Sigma^{f,\tau}:=\Sigma\times \mathbb{Z}^d\times \mathbb{R}/\sim,
\end{equation*}
where $f:\Sigma\to \mathbb{Z}^d$ records the change of $\mathbb{Z}^d$-coordinate under the return return map. 

The proof of Theorem \ref{main thm} involves investigating a two-parameter twisted transfer operator 
\begin{align*}
&L_{1-i\alpha,\underline{\theta}}: C(\Sigma^+,\mathbb{C})\to C(\Sigma^+,\mathbb{C}),\\
 &L_{1-i\alpha, \underline{\theta}}F(x)=\sum_{\sigma y=x}e^{-(1-i\alpha)\tau(y)+\langle f(y),\underline{\theta}\rangle}F(y),
\end{align*}
 where $\alpha\in \mathbb{R}$, $\underline{\theta}\in \mathbb{T}^d$ ($\mathbb{R}$ and $\mathbb{T}^d$ appear  as they are the unitary duals of $\mathbb{R}$ and $\mathbb{Z}^d$ respectively), and $(\Sigma^+,\sigma)$ is the one-sided version of $(\Sigma,\sigma)$. Local mixing is related with the asymptotic expansion of $\sum_{n=0}^{\infty}L^n_{1-i\alpha,\underline{\theta}}$ around $(0,\underline{0})$, so a knowledge on  the regularity of $L_{1-i\alpha, \underline{\theta}}$ and its eigenvalue with maximal modulus $\lambda_{1-i\alpha,\underline{\theta}}$ is necessary. When $M_0$ is compact, $L_{1-i\alpha,\underline{\theta}}$ and $\lambda_{1-i\alpha,\theta}$ are analytic with respect to $(\alpha,\underline{\theta})$. In the current setting, the set $\Sigma$ comes from a Poincar\'{e} section of the geodesic flow over $\operatorname{T}^1(M_0)$. It may happen that some vector goes far into the cusps before it returns to the Poincar\'{e} section. When lifting the coding to  the covering space, this may result that the change of $\mathbb{Z}^d$-coordinate under the first return map is large and hence $L_{1-i\alpha,\underline{\theta}}$ may not analytic. Here one looks deep into the relation between the function $f$ and the amount of time a vector spends in cusps and we benefits from the thorough analysis carried out by Ledrappier-Sarig \cite{LS}. This resonates with Epstein's work \cite{Ep} in some degree. One main part of  \cite{Ep} is about the differentiability of the eigenvalue $\lambda_{\underline{\theta}}$ of the corresponding Laplace operator with respect to $\underline{\theta}\in \mathbb{T}^d$.

\subsection{On the generalization of Theorem \ref{main thm}}
Let $\mathbb{H}^{n}$ be the hyperbolic $n$-space with $n> 1$. Let $\Gamma_0<\operatorname{Isom}_{+}(\mathbb{H}^n)$ be a Schottky group and $\Gamma$ be a normal subgroup of $\Gamma_0$ satisfying $\Gamma\backslash \Gamma_0\cong \mathbb{Z}^d$. We expect that  the geodesic flow on $\operatorname{T}^1(\Gamma\backslash \mathbb{H}^n)$ can be shown to satisfy the local mixing property using the argument of Theorem \ref{main thm} and the knowledge of the corresponding twisted transfer operator obtained by Babillot and Peign\'{e} \cite{BP}.

\subsection{Acknowledgement} I would like to thank my advisor Hee Oh, who was  a source of inspiration and encouragement throughout. 
\section{Preparations: coding and transfer operator}

\subsection{Symbolic model for the geodesic flow}
Stadlbauer \cite{Sta} and Ledrappier-Sarig \cite{LS} use a two-sided countable Markov shift $(\Sigma,\sigma)$ with the set of countable states $\mathcal{A}$ to encode the nonwondering set of the geodesic flow on $\T^1(M_0)$. We provide a sketch of the construction here. We start with recalling some basic notions regarding \emph{countable Markov shifts}. Suppose $\mathcal{A}$ is a \emph{countable} set of states and $\mathbb{A}=(t_{ab})_{\mathcal{A}\times \mathcal{A}}$ is a matrix of zeros and ones with no columns or rows which are made solely of zeroes.

\begin{defn}[TMS] The two-sided \emph{topological Markov shift (TMS)} with the set of states $\mathcal{A}$ and transition matrix $\mathbb{A}$ is the set
\begin{equation*}
\Sigma:=\{(\ldots, x_{-1},x_0,x_1,\ldots)\in \mathcal{A}^{\mathbb{Z}}:\forall i\in \mathbb{Z},t_{x_i x_{i+1}}=1\}
\end{equation*}
equipped with the topology generated by the two sided cylinders
\begin{equation*}
[a_m,\ldots,a_{n}]:=\{x\in \Sigma:(x_m,\ldots,x_{n})=(a_m,\ldots,a_{n})\}\,\,\,(m,n\in \mathbb{Z}\,\,\,\text{and}\,\,\,m<n)
\end{equation*}
and the action of the left shift $\sigma:\Sigma\to \Sigma,(\sigma x)_i=x_{i+1}$.
\end{defn}
We say $(\Sigma, \sigma)$ has the \emph{Big Images and Preimages (BIP) property} if there is a finite set of states $b_1,\ldots,b_N$ s.t. $\forall a\in \mathcal{A},\,\exists 1\leq i,j\leq N$ s.t. $t_{b_ia}t_{ab_j}=1$.

Given a TMS $(\Sigma, \sigma)$, a function $g:\Sigma\to \mathbb{R}$ is called \emph{locally H\"{o}lder continuous} if $\exists C>0,0<\theta<1$ s.t. 
\begin{equation*}
x_{-n}^{n}=y_{-n}^{n}\,\,\,(n\geq 0)\implies |g(x)-g(y)|<C\theta^n.
\end{equation*}

{\textbf{Stadlbauer-Ledrappier-Sarig coding.}}
As $M_0$ contains at least one cusp, Tukia \cite{Tu} showed that $\Gamma_0$ has a special fundamental domain $D_0$, containing the origin $0$ in the Poincar\'{e} disk model, all of whose vertices lie in $\partial \mathbb{H}^2$, so that the set of side pairings $\mathcal{S}$ of $D_0$ is a minimal symmetric generating set for $\Gamma_0$. The classical Bowen-Series coding $(\tilde{\Sigma}, \tilde{\sigma})$ is constructed using an edge cutting sequence $(x_k)$ of a bi-infinite geodesic so that the geodesic is entering the translate $\gamma_k D_0$ as it passes through the edge $x_k$. 

But the classical coding is difficult to use in the presence of cusps. One of the issues is that the first return time map, which gives the amount of time a bi-infinite geodesic spends in $D_0$, is not H\"{o}lder or even bounded because one can find a geodesic which travels arbitrarily far in the cusp.

The rough idea of the new coding is that one considers a subset $A$ in $\tilde{\Sigma}$ whose elements represent unit tangent vectors which do not travel far into the cusps of $D_0$ in the future and past, and then records the first return map of $A$. So one needs to clump together all large powers of parabolic elements in the cutting sequences. 

A more detailed description is as follows. Let $\mathcal{C}$ be the collection of all minimal length freely reduced words in $\mathcal{S}$ representing parabolic elements.  They choose a sufficiently large even number $2N$ so that the length of every element of $\mathcal{C}$ divides $2N$ and let $\mathcal{C}^*$ be the collection of powers of elements of $\mathcal{C}$ of length $2N$. A precise description for the set $A$ is that $A$ consists of the sequences $(b_i)_{i\in \mathbb{Z}}$ such that $(b_{-N+1},\ldots,b_N)\notin \mathcal{C}^*$. Let $\mathcal{A}_1$ be the set containing all the strings $(b_0,b_1,\ldots,b_{2N})$ so that $b_0b_1\ldots b_{2N}$ is reduced in $\Gamma_0$ and so that neither $b_0b_1\ldots b_{2N-1}$ nor $b_1b_2\ldots b_{2N}$ lies in $\mathcal{C}^*$. Let $\mathcal{A}_2$ be the set containing the strings of the form $(b, w^s,w_1,\ldots, w_{k-1},c)$ where $b\in \mathcal{S}-\{w_{2N}\}$, $w=w_1\ldots w_{2N}\in \mathcal{C}^*$, $w_i\in \mathcal{S}$ for all $i$, $1\leq k\leq 2N$, $s\geq 1$ and $c\in \mathcal{S}-\{w_k\}$.

Let $\mathcal{A}=\mathcal{A}_1\cup \mathcal{A}_2$. Given a word $x=(b_i)_{i\in \mathbb{Z}}\in A$, we explain how to rewrite it using the alphabets in $\mathcal{A}$. If $(b_{-N+1},\ldots,b_0,\ldots,b_N,b_{N+1})\in \mathcal{A}_1$, then $x_0=(b_{-N+1},\ldots, b_0,\ldots, b_N,b_{N+1})$ and we shift $x$ leftward by $1$ to compute $x_1$. If not, let $x_0$ be the unique subsequence in $x$ which contains $(b_{-N+1},\ldots,b_N)$ and is an element of $\mathcal{A}_2$. Then $x_0$ is of the form $(b_{-N+1},w^{s},w_1,\ldots, w_{k-1},b_{M})$ for some $w\in \mathcal{C}^*$ and $M>N+1$. We shift $x$ leftward by $2N(s-1)+k+1$ to compute $x_1$. To get $x_{-1}$, we consider $(b_{-N},\ldots,b_N)$. If $(b_{-N},\ldots, b_{N})\in \mathcal{A}_1$, then $x_{-1}=(b_{-N},\ldots,b_{N})$. If not, let $x_{-1}$ be the unique subsequence in $x$ which contains $(b_{-N},\ldots,b_N)$ and is an element of $\mathcal{A}_2$. In this way, we have defined the two-sided Markov shift $(\Sigma, \sigma)$ with countable states for the dynamical system $(\T^1(M_0),\mathcal{G}^s)$.

Let $p:\operatorname{T}^1(M)\to \operatorname{T}^1(M_0)$ be the covering map. Let $\Omega_0\subset \operatorname{T}^1(M_0)$ be the nonwondering set for the geodesic flow and set $\Omega:=p^{-1}(\Omega_0)$. 
Let $\{D_{\underline{\xi}}:\underline{\xi}\in \mathbb{Z}^d\}$ be the group of deck transformations for  $p:\operatorname{T}^1(M)\to \operatorname{T}^1(M_0)$ and we enumerate them in such a way that $D_{\underline{\xi}+\underline{\xi}'}=D_{\underline{\xi}}\circ D_{\underline{\xi}'}$ for any $\underline{\xi},\underline{\xi}'\in \mathbb{Z}^d$.  We obtain a coding for the geodesic flow $\mathcal{G}^s:\Omega\to \Omega$ from Stadlbauer-Ledrappier-Sarig coding by adding a function that records the $\mathbb{Z}^d$-displacement under the first return time map.
\begin{lem}\cite[Lemma 2.2]{LS}
\label{subshift}
There exist a topologically mixing TMS $(\Sigma,\sigma)$ of countable states, a positive continuous function $\tau:\Sigma\to \mathbb{R}$ and a continuous function $f:\Sigma\to \mathbb{Z}^d$ s.t. $f(x)=f(x_0)$ with the following properties:
\begin{enumerate}
\item $(\Sigma,\sigma)$ has the BIP property.

\item The quotient space $\Sigma^{f,\tau}:=\Sigma\times \mathbb{Z}^d\times \mathbb{R}/(x,\underline{\xi},t)\sim (\sigma x, \underline{\xi}+f(x),t-\tau(x))$ is homeomorphic to $\Omega$. Denote the homeomorphism by $\pi$.

\item The suspension flow over $\Sigma^{f,\tau}$ is conjugate to the geodesic flow $\mathcal{G}^s: \Omega \to \Omega$. 

\item There exist a locally H\"{o}lder continuous function $r:\Sigma\to \mathbb{R}$ which depends only on nonnegative coordinates and a uniformly continuous function $h:\Sigma\to \mathbb{R}$ such that $r=\tau-(h-h\circ \sigma)$.

\item There exist $C>0$ and $K>0$ s.t. $r+r\circ \sigma+\cdots+r\circ \sigma^{n-1}\geq C$ for all $n\geq K$.

\item If we set $Q_{\underline{\xi_0},t_0}(x,\underline{\xi},t)=(x,\underline{\xi}+\underline{\xi_0},t+t_0)$, then $\pi\circ Q_{\underline{\xi_0},t_0}=\big(\mathcal{G}^{t_0}\circ D_{\underline{\xi_0}}\big)\circ \pi$ for all $(\underline{\xi_0},t_0)\in \mathbb{Z}^d\times \mathbb{R}$.
\end{enumerate}
\end{lem}

\subsection{The transfer operator and the Haar measure}
In this subsection, we charaterize the Haar measure $\mathsf{m}^{\operatorname{Haar}}$ on $\T^1(M)$ in the symbolic model $\Sigma^{f,\tau}$. 

Let $(\Sigma^+,\sigma)$ be the one-sided version of the Markov shift $(\Sigma,\sigma)$ given in Lemma \ref{subshift}. Here we use $\sigma$ to denote the left-shift map on $\Sigma^+$ by abusing notation. Let $r$ be the function on $\Sigma$ given in Lemma \ref{subshift}. As it depends only on nonnegative coordinates, we can regard it as a function on $\Sigma^+$. For a map $g$ on $\Sigma^+$ and $n\geq 1$, we write
\begin{equation*}
g_n(x)=g(x)+g(\sigma (x))+\cdots+g(\sigma^{n-1}(x)).
\end{equation*}
The \emph{Gurevich topological pressure} of $-r$ is given by
\begin{equation*}
P_{\operatorname{top}}(-r):=\lim_{n\to \infty} \frac{1}{n}\log \sum_{\sigma^n z=z} e^{-r_n(z)}1_{[a_0]}(z)\,\,\,\text{for some state}\,\,\,a_0\in \mathcal{A}.
\end{equation*}
The limit exists and is independent of the choice of $a_0\in\mathcal{A}$ \cite[Proposition 3.2]{Sa2}.

Denote by $C_{B}(\Sigma^+)$ the set of bounded continuous functions. The Ruelle transfer operator $L:C_B(\Sigma^+)\to C_B(\Sigma^+)$ is defined by
\begin{equation*}
(LF)(x)=\sum_{\sigma (y)=x}e^{-r(y)}F(y)\,\,\,\text{for any}\,\,\,F\in C_B(\Sigma^+).
\end{equation*}

As we are working with a countable Markov shift, we need a few words to justify that the Ruelle transfer operator $L$ is well-defined.

\begin{lem}
\label{justify}
The transfer operator $L:C_{B}(\Sigma^+)\to C_B(\Sigma^+)$ is well-defined.
\end{lem}

\begin{proof}
Using Stadlbauer-Ledrappier-Sarig coding, we can code the geodesic flow  over $\operatorname{T}^1(M_0)$  using the suspension flow over the suspension space $\Sigma^{\tau}:=\Sigma\times \mathbb{R}/(x,t)\sim (\sigma x,t-\tau(x))$. Denote the map by $\pi_A$. Define $\psi:\Sigma^+\to \mathbb{R}_{>0}$ by 
\begin{equation}
\label{eigenfunction}
\psi(x)=\text{hyperbolic length measure of}\,\,\,\{\pi_A(y,h(y)):y_0^{\infty}=x_0^{\infty}\}. 
\end{equation}
It is a continuous function.  Set $P(x_0)=\{a\in \mathcal{A}:\,[a,x_0]\neq \emptyset\}$. Observe that
\begin{align*}
&\{\pi_A(y,h(y)):y_0^{\infty}=x_0^{\infty}\}=\pi_A\{\sqcup_{a\in P(x_0)}\{(y,h(y)):y_{-1}^{\infty}=(a,x_0,x_1,\ldots)\}\}\\
=&\pi_A\{\sqcup_{a\in P(x_0)}\{(z,\tau(z)+h(\sigma(z))):z_0^{\infty}=(a,x_0,x_1,\ldots)\}\}\\
=&\pi_A\{\sqcup_{a\in P(x_0)}\{(z,r(z)+h(z)):z_0^{\infty}=(a,x_0,x_1,\ldots)\}\}\\
=&\sqcup_{a\in P(x_0)}\mathcal{G}^{r(ax_0^{\infty})}\pi_A\{(z,h(z)):z_0^{\infty}=(a,x_0,x_1,\ldots)\}.
\end{align*}
So we have the following functional equation for $\psi$:
\begin{equation}
\label{eigenfunction}
L\psi=\psi.
\end{equation}

Meanwhile, for any $a_0\in \mathcal{A}$, $x\in [a_0]$ and $n\geq 1$, we define the following map
\begin{align*}
\{y\in [a_0]: \sigma^n y=x\}&\to \{z\in [a_0]:\sigma^n z=z\}\\
y&\mapsto (y_0,y_1,\ldots, y_{n-1},y_0,y_1,\ldots, y_{n-1},\ldots).
\end{align*}
This is a bijection. As $r$ is locally H\"{o}lder, there exists $M>0$ such that for any $a_0\in\mathcal{A}$, $x\in [a_0]$ and $n\in \mathbb{N}$, we have 
\begin{equation}
\label{period}
\sum_{\sigma^n z=z} e^{-r_n(z)}1_{[a_0]}(z)=M^{\pm 1}(L^n 1_{[a_0]})(x). 
\end{equation}
As $\psi$ is continuous, one is able to  compare $1_{[a_0]}(y)$ and $\psi(y)$ for $y\in [a_0]$. Then \eqref{eigenfunction} and \eqref{period} together imply $P_{\operatorname{top}}(-r)<\infty$. Then we have that $\psi$ is locally 	H\"{o}lder continuous and bounded away from zero and infinity \cite[Corollary 2]{Sa}. This enables us to  compare any bounded function $F$ with $\psi$ and hence deduce that the operator $L:C_B(\Sigma^+)\to C_B(\Sigma^+)$ is well-defined.
\end{proof}

\begin{lem}
\cite[Lemma 3.1]{LS}
\label{measure}
There exists a unique finite measure $\rho$ on $\Sigma^+$ such that $\rho (LF)=\rho (F)$ for all non-negative $F\in C_B(\Sigma^+)$, and such that $d\nu^+:=\psi d\rho$  with $\psi$ given as in \eqref{eigenfunction} is a shift invariant probability measure on $\Sigma^+$. Let $\nu$ be the natural extension to $\Sigma$. 
Then the Haar measure $\mathsf{m}^{\operatorname{Haar}}$ on $\operatorname{T}^1(M)$ is equal to $\frac{\mathsf{m}_0}{\int \tau d\nu}\big(d\nu d\underline{\xi} dt|_{\{(x,\underline{\xi},t):0\leq t< \tau(x)\}}\big)\circ \pi^{-1}$, where $\mathsf{m_0}$ is the constant given in (\ref{area}).
\end{lem}
In fact, we have that $(\Sigma^+,\sigma)$ is a topologically mixing Markov shift satisfying the BIP property, $r$ is locally H\"{o}lder and the Gurevich pressure $P_{\operatorname{top}}(-r)<\infty$. The existence of $\rho$ and the statement about $\nu^+$ follow from the general theory of the BIP shifts \cite{Sa}. Using the argument in \cite[Section 6]{BL2},  one obtains the symbolic description of $\mathsf{m}^{\operatorname{Haar}}$.

\section{Local mixing and matrix coefficients for local functions} 


\subsection{Asymptotic analysis of symbolic sum}
\label{analysis}

For any $x,y\in \Sigma^+$, let $t(x,y):=\min \{n\geq 0: x_n\neq y_n\}$. 
Define for $\Phi:\Sigma^+\to \mathbb{R}$, $$\operatorname{Lip}(\Phi):=\sup \{|\Phi(x)-\Phi(y)|/\theta^{t(x,y)}:x_0=y_0\}.$$  Set\begin{equation*}
\mathcal{L}(\Sigma^+):=\{\Phi:\Sigma^+\to \mathbb{R}: \lVert \Phi \rVert:=\operatorname{Lip}(\Phi)+\lVert \Phi\rVert_{\infty}<\infty\}.
\end{equation*}
This is  a Banach space.

Let $r$ be the function on $\Sigma^+$ given as in Lemma \ref{subshift}. For any $t>1$, $\Phi\in \mathcal{L}(\Sigma^+)$ and $u\in C_c(\mathbb{R})$, consider the symbolic sum (which is a function on $\Sigma^+\times \mathbb{Z}^d$):
\begin{equation*}
\mathsf{Q}_t(\Phi \otimes u)(x,\underline{\xi}):= \sum_{n=0}^{\infty} \sum_{\sigma^n y=x} e^{-r_n(y)}(\Phi \cdot \psi)(y) \delta_{\underline{\xi}}(f_n(y)) u(r_n(y)-t).
\end{equation*}

\begin{thm}\cite[Lemma 5.1]{LS}
\label{asym}
For any $(x,\underline{\xi})\in \Sigma^+\times \mathbb{Z}^d$, we have 
\begin{equation*}
\lim_{t\to \infty} t^{p+h/2} \mathsf{Q}_t(\Phi\otimes u)(x,\underline{\xi})=\frac{\mathsf{c}\,\mathsf{m}_0}{\int \tau d\nu}\psi(x) \nu(\Phi)\int u(t)dt,
\end{equation*}
where $p,h,\mathsf{c}$ and $\mathsf{m}_0$ are the constants given in (\ref{con 1}), (\ref{con 2}) and (\ref{area}) and the convergence is uniform on compact sets in $\Sigma^+\times \mathbb{Z}^d$.
\end{thm}


This is shown in \cite{LS} for special functions and the way they formulated the result is different than ours. To state their result, recall the decomposition $\mathbb{R}^d=\mathbb{E}_{\text{p}}\oplus \mathbb{E}_{\text{h}}$ defined in Section \ref{subspace}.
Define functions $F_{\text{p}}$ on $\mathbb{E}_{\text{p}}$ and $F_{\text{h}}$ on $\mathbb{E}_{\text{h}}$ by
\begin{align*}
F_{\text{p}}(\underline{\xi}_{\text{p}})=\frac{1}{(2\pi)^p}\int_{\mathbb{E}_{\text{p}}} e^{i\langle \underline{\xi}_{\text{p}},\underline{x}\rangle}d\underline{x},\\
F_{\text{h}}(\underline{\xi}_{\text{h}})=\frac{1}{(2\pi)^h}\int_{\mathbb{E}_{\text{h}}} e^{i\langle \underline{\xi}_{\text{h}},\underline{y}\rangle}d\underline{y}.
\end{align*}
Both functions are positive, uniformly continuous, and absolutely integrable. Set $F(\underline{\xi}_{\text{p}}+\underline{\xi}_{\text{h}})=F_{\text{p}}(\underline{\xi}_{\text{p}})F_{\text{h}}(\underline{\xi}_{\text{h}})$. For every $\epsilon>0$, one can construct two positive, uniformly continuous, bounded functions $F^{\pm}_{\epsilon}$  such that for all $\underline{\xi}(=\underline{\xi}_{\text{p}}+\underline{\xi}_{\text{h}})\in \mathbb{R}^d$ and $e^{-\epsilon}<t_1,t_2<e^{\epsilon}$,
\begin{equation*}
F^{-}_{\epsilon}(\un{\xi}_{\text{p}}+\un{\xi}_{\text{h}})\leq F(t_1\un{\xi}_{\text{p}}+t_2\un{\xi}_{\text{h}})\leq F^+_{\epsilon}(\un{\xi}_{\text{p}}+\un{\xi}_{\text{h}})
\end{equation*}
and such that $F^+_{\epsilon}/F^-_{\epsilon}\to1$ uniformly on compact sets as $\epsilon \to 0^+$. Let $1_{[z_0]}$ be the characteristic function of a cylinder set $[z_0]$ in $\Sigma^+$ and $1_{[-a/2,a/2]}$ be the characteristic function of an interval $[-a/2,a/2]$ in $\mathbb{R}$. It is shown in \cite[Lemma 5.1]{LS} that for any $\epsilon>0$, there exists $t_0>0$ such that for any $t>t_0$, $x\in \Sigma^+$ and $\un{\xi}\in \mathbb{Z}^d$, if $\un{\xi}=\un{\xi}_{\text{p}}+\un{\xi}_{\text{h}}$ where $\un{\xi}_{\text{p}}\in \mathbb{E}_{\text{p}}$ and $\un{\xi}_{\text{h}}\in \mathbb{E}_{\text{h}}$, then
\begin{align*}
& t^{p+h/2}\mathsf{Q}_t(1_{[z_0]} \otimes 1_{[-a/2,a/2]})(x,\underline{\xi})\leq e^{\epsilon} \left[F^+_{\epsilon}\left(\frac{\un{\xi}_{\text{p}}}{t}+\frac{\un{\xi}_{\text{h}}}{\sqrt{t}}\right)+\epsilon\right]\frac{\nu[z_0] a}{\int \tau d\nu} \psi(x);\\
& t^{p+h/2}\mathsf{Q}_t(1_{[z_0]} \otimes 1_{[-a/2,a/2]})(x,\underline{\xi})\geq e^{-\epsilon} \left[F^-_{\epsilon}\left(\frac{\un{\xi}_{\text{p}}}{t}+\frac{\un{\xi}_{\text{h}}}{\sqrt{t}}\right)+\epsilon\right]\frac{ \nu[z_0] a}{\int \tau d\nu} \psi(x).
\end{align*}
One of the key observations of the proof is that applying the Fourier inversion theorem to $1_{[-a/2,a/2]}$ and $\delta_{\underline{\xi}}$, one can express $\mathsf{Q}_t$ in terms of the sum of two-parameter twisted transfer operators $\sum_{n=0}^{\infty}L^n_{1-i\alpha,\underline{\theta}}(\psi 1_{[z_0]})$ with $(\alpha,\underline{\theta})\in \mathbb{R}\times \mathbb{T}^d$, where $L_{1-i\alpha, \underline{\theta}}(\psi 1_{[z_0]})(x)=\sum_{\sigma y=x}e^{-(1-i\alpha)\tau(y)+\langle f(y),\underline{\theta}\rangle}\psi(y) 1_{[z_0]}(y)$. The only singularity is the point $(0,\underline{0})\in \mathbb{R}\times \mathbb{T}^d$ and one carries out a detailed analysis on the asymptotic expansion of  $\sum_{n=0}^{\infty}L^n_{1-i\alpha,\underline{\theta}}(\psi 1_{[z_0]})$ at the singularity.

The statement of Theorem \ref{asym} for special functions then can be obtained using the properties of $F^{\pm}_{\epsilon}$. The argument of the proof of \cite[Lemma 5.1]{LS} also works for general functions so we obtain Theorem \ref{asym} for general functions.

\subsection{Correlation functions for $(\Gamma\backslash G, a_s, \mathsf{m}^{\operatorname{Haar}})$}
We write
\begin{align*}
&\tilde{X}:=\Sigma\times \mathbb{Z}^d\times \mathbb{R},\\
&\tilde{X}^+ :=\Sigma^+\times \mathbb{Z}^d\times \mathbb{R}.
\end{align*}
Consider the product measure on $\tilde{X}$:
\begin{equation*}
d\tilde{\mathsf{M}}:=\frac{\mathsf{m}_0}{\int \tau d\nu}(d\nu d\underline{\xi} ds).
\end{equation*}
The Haar measure $\mathsf{m}^{\operatorname{Haar}}$ on $\Gamma\backslash G$ corresponds to the measure on $\Sigma^{f,\tau}$ induced by $\tilde{\mathsf{M}}$.
For $\Psi_1,\Psi_2\in C_c(\tilde{\mathsf{X}}^+)$, define
\begin{equation*}
I_t(\Psi_1,\Psi_2):=\sum_{n=0}^{\infty} \int_{\tilde{\mathsf{X}}} \Psi_1\circ \tilde{\zeta}^n(x,\underline{\xi},s+t)\cdot \Psi_2(x,\underline{\xi},s)d\tilde{\mathsf{M}}(x,\underline{\xi},s)
\end{equation*}
where $\tilde{\zeta}^n(x,\underline{\xi},s):=(\sigma^nx, \underline{\xi}+f_n(x),s-r_n(x))$. It follows from the property of $r$ (Lemma \ref{subshift} (5)) that, for any $(x,\underline{\xi},s)\in \operatorname{supp} \Psi_2$,
$$\Psi_1\circ \tilde{\zeta}^n(x,\underline{\xi},s+t)=0$$
for $n$ large enough. So $I_t(\Psi_1,\Psi_2)$ is in fact a finite sum. 

\begin{defn}
Let $\mathcal{F}_0$ be the family of functions on $\tilde{X}^+$ of the form
\begin{equation*}
\Psi(x,\underline{\xi},s)=\Phi(x)\delta_{\underline{\xi_0}}(\underline{\xi})u(s)
\end{equation*}
where $\Phi\in \mathcal{L}(\Sigma^+)$, $u\in C_c(\mathbb{R})$ and $\underline{\xi_0}\in \mathbb{Z}^d$. Denote by $\mathcal{F}$ the space of functions which are finite linear combinations of functions from $\mathcal{F}_0$.
\end{defn}

\begin{lem}
\label{translate}
Let $\Psi_2(x,\underline{\xi},s)=\Phi(x)\delta_{\underline{\xi_0}}(\underline{\xi})u(s)\in \mathcal{F}_0$. Then for any $\Psi_1\in C_c(\tilde{\mathsf{X}}^+)$, we have
\begin{equation*}
I_t(\Psi_1,\Psi_2)=\frac{\mathsf{m}_0}{\int \tau d\nu}\int_{\tilde{\mathsf{X}}^+}\Psi_1(x,\underline{\xi},s)\cdot \mathsf{Q}_{t-s}(\Phi \otimes u)(x,\underline{\xi}-\underline{\xi_0})d\rho(x) d\underline{\xi} ds.
\end{equation*}
\end{lem}

\begin{proof}
Since $d\rho$ is an eigenmeasure of $L$ with eigenvalue $1$, for any $F,G\in \mathcal{L}(\Sigma^+)$, we have the following equality
\begin{equation*}
\int_{\Sigma^+} F\circ \sigma \cdot G d\rho= \int_{\Sigma^+} L(F\circ \sigma \cdot G) d\rho=\int_{\Sigma^+} F\cdot (L G)d\rho.
\end{equation*}
Using this and the Fubini theorem, we obtain
\begin{align*}
&I_t(\Psi_1,\Psi_2)\\
=&\int_{\tilde{\mathsf{X}}^+}\Psi_1(x,\underline{\xi},s)\sum_{n=0}^{\infty}\sum_{\sigma^n y=x}e^{-r_n(y)} (\Phi\cdot \psi)(y)\delta_{\underline{\xi_0}}(\underline{\xi}-f_n(y)) u(s-t+r_n(y))d\rho(x) d\underline{\xi} ds\\
=& \frac{\mathsf{m}_0}{\int \tau d\nu}\int_{\tilde{\mathsf{X}}^+}\Psi_1(x,\underline{\xi},s)\mathsf{Q}_{t-s}(\Phi\otimes u)(x,\underline{\xi}-\underline{\xi_0})d\rho(x)d\underline{\xi} ds.
\end{align*}
\end{proof}

\begin{proof}[Proof of Theorem \ref{main thm}]
Without loss of generality, we may suppose that $\psi_1$ and $\psi_2$ are defined on the suspension space $\Sigma^{f,\tau}$, continuous and compactly supported. For each $i=1,2$, let $\tilde{\Psi}_i\in C_c(\tilde{\mathsf{X}})$ be the lift of $\psi_i$ to $\tilde{\mathsf{X}}$ satisfying
\begin{equation*}
\psi_i[(x,\underline{\xi},s)]=\sum_{n\in \mathbb{Z}} \tilde{\Psi}_i\circ \zeta^n(x,\underline{\xi},s),
\end{equation*}
with $\zeta(x,\underline{\xi},s):=(\sigma x,\underline{\xi}+f(x),s-\tau(x))$.

Using the unfolding, 
\begin{align}
\label{eq 1}
&\int_{\Gamma\backslash G}\psi_1(xa_t)\psi_2(x)d\mathsf{m}^{\operatorname{Haar}}(x)\nonumber\\
=&\sum_{n\in\mathbb{Z}} \int_{\tilde{\mathsf{X}}}\tilde{\Psi}_1\circ \zeta^n(x,\underline{\xi},s+t)\cdot \tilde{\Psi}_2(x,\underline{\xi},s)d\tilde{\mathsf{M}}\nonumber\\
=&\sum_{n=0}^{\infty} \int_{\tilde{\mathsf{X}}} \tilde{\Psi}_1\circ \zeta^n(x,\underline{\xi},s+t)\cdot \tilde{\Psi}_2(x,\underline{\xi},s)d\tilde{\mathsf{M}}\nonumber\\
+&\sum_{n=1}^{\infty}\int_{\tilde{\mathsf{X}}} \tilde{\Psi}_1\circ \zeta^{-n}(x,\underline{\xi},s+t)\cdot \tilde{\Psi}_2(x,\underline{\xi},s)d\tilde{\mathsf{M}}.\nonumber
\end{align}
For the second term in the above equation, note that for $(x,\xi,s)\in \operatorname{supp}(\tilde{\Psi}_2)$,
\begin{equation*}
\tilde{\Psi}_1\circ \zeta^{-n}(x,\underline{\xi},s+t)=\tilde{\Psi}_1(\sigma^{-n}x,\underline{\xi}-f_{-n}(x),s+t+\tau_{-n}(x)),
\end{equation*}
where $f_{-n}(x)=\sum_{k=1}^{n}f(\sigma^{-k}x)$ and $\tau_{-n}(x)=\sum_{k=1}^{n}\tau (\sigma^{-k}x)$. It tends to $0$ when $t$ is large enough, as $\tau_n(x)>0$. Define
\begin{equation*}
\Psi_i(x,\underline{\xi},s):=\tilde{\Psi}_i(x,\underline{\xi},s-h(x))\,\,\,\text{for}\,\,i=1,2.
\end{equation*}
As $\tilde{\Psi}_i$ is compactly supported, the new function $\Psi_i$ is also a continuous function with compact support. Moreover, we have $\tilde{\mathsf{M}}(\Psi_i)=\tilde{\mathsf{M}}(\tilde{\Psi}_i)$. The first term of the last equation above equals $I_t(\Psi_1,\Psi_2)$. 
 Theorem \ref{main thm} follows if this following holds
\begin{equation}
\label{case}
\lim_{t\to \infty} t^{p+h/2} I_t(\Psi_1,\Psi_2)=\mathsf{c}\,\tilde{\mathsf{M}}(\Psi_1)\tilde{\mathsf{M}}(\Psi_2).
\end{equation}

The proof of (\ref{case}) is divided into the following three cases.

\textbf{\textit{Case 1}}. Assume $\Psi_1\in C_c(\tilde{\mathsf{X}}^+)$ and $\Psi_2\in \mathcal{F}$. It suffices to consider the case where
\begin{equation*}
\Psi_2(x,\underline{\xi},s)=\Phi(x)\delta_{\underline{\xi_0}}(\underline{\xi})u(s)\in \mathcal{F}_0.
\end{equation*}
By Lemma \ref{translate},
\begin{equation*}
I_t(\Psi_1,\Psi_2)=\frac{\mathsf{m}_0}{\int \tau d\nu}\int_{\tilde{\mathsf{X}}^+} \Psi_1(x,\underline{\xi},s)\cdot \mathsf{Q}_{t-s}(\Phi\otimes u)(x,\underline{\xi}-\underline{\xi_0})d\rho(x)d\underline{\xi} ds.
\end{equation*}
Applying Theorem \ref{asym} to the integrand, we obtain
\begin{align*}
&\lim_{t\to \infty} t^{p+h/2} \Psi_1(x,\underline{\xi},s)\cdot \mathsf{Q}_{t-s}(\Phi\otimes u)(x,\underline{\xi}-\underline{\xi_0})\\
=&\frac{\mathsf{c}}{\int \tau d\nu} \tilde{\mathsf{M}}(\Psi_2)\Psi_1(x,\underline{\xi},s)\psi(x).
\end{align*}
As the above convergence is uniform on compact sets, we have
\begin{equation*}
\lim_{t\to \infty}t^{p+h/2} I_t(\Psi_1,\Psi_2)=\mathsf{c}\, \tilde{\mathsf{M}}(\Psi_1) \tilde{\mathsf{M}}(\Psi_2).
\end{equation*}

\textbf{\textit{Case 2}}. Assume $\Psi_1,\Psi_2\in C_c(\tilde{\mathsf{X}}^+)$. Approximate $\Psi_2$ by functions in $\mathcal{F}$: it follows from the Stone-Weierstrass theorem that for any $\epsilon>0$, there exists $F_2\in \mathcal{F}$ such that for any $(x,\underline{\xi},s)\in \tilde{\mathsf{X}}^+$
$$ |\Psi_2(x,\underline{\xi},s)-F_2(x,\underline{\xi},s)|<\epsilon.$$ 
We can find $w_2\in \mathcal{F}$ satisfying
\begin{itemize}
\item  $ |\Psi_2(x,\underline{\xi},s)-F_2(x,\underline{\xi},s)|<\epsilon \,w_2(x,\underline{\xi},s)$  for any $(x,\underline{\xi},s)\in \tilde{\mathsf{X}}^+$;

\item $\tilde{\mathsf{M}}(w_2)$ is bounded by the size of the $\mathbb{Z}^d\times \mathbb{R}$-support of $\Psi_2$.
\end{itemize}
As (\ref{case}) holds for the pairs $(\Psi_1,F_2),(\Psi_1,w_2)$ and $\epsilon$ is arbitrary, it also holds for the pair $(\Psi_1,\Psi_2)$.

\textbf{\textit{Case 3}}. Let $\Psi_1,\Psi_2\in C_c(\tilde{\mathsf{X}})$. For any $\epsilon>0$, there exist $k\in \mathbb{N}$ and $F_i\in C_c(\tilde{\mathsf{X}}^+)$ such that
$$\lvert \Psi_i\circ \tilde{\zeta}^k(x,\underline{\xi},t)-F_i(x,\underline{\xi}+f_k(x),t-r_k(x))\rvert<\epsilon. $$
We can find $w_i\in \mathcal{F}$ satisfying
\begin{itemize}
\item $\lvert \Psi_i\circ \tilde{\zeta}^k(x,\underline{\xi},t)-F_i(x,\underline{\xi}+f_k(x),t-r_k(x))\rvert<\epsilon \cdot w_i(x,\underline{\xi}+f_k(x),s-r_k(x))$;
\item $\tilde{\mathsf{M}}(\lvert w_i\rvert)$ are bounded by the size of $\mathbb{Z}^d\times \mathbb{R}$-support of $\Psi_i$.
\end{itemize}
Since $\tilde{\mathsf{M}}$ is invariant under the action of $\tilde{\zeta}$, we can replace $\Psi_i$ by $\Psi_i\circ \tilde{\zeta}^k$ in (\ref{case}). For (\ref{case}) holds for the pairs $(F_i,F_j), (F_i,w_j),(w_i,w_j)$ with $i,j\in\{1,2\}$, it is valid for the pair $(\Psi_1\circ \tilde{\zeta}^k, \Psi_2\circ \tilde{\zeta}^k)$.

\end{proof}

 
\end{document}